\documentclass[12pt,twoside]{amsart}

\usepackage{amsmath,amsthm,amscd,amssymb,mathrsfs,graphicx,amsfonts,mathrsfs}
\usepackage{amsfonts}
\usepackage{amssymb,enumerate}
\usepackage{amsthm}
\usepackage[all]{xy}
\usepackage{hyperref}
\allowdisplaybreaks[4] \footskip=12pt
\renewcommand{\uppercasenonmath}[1]{}

\numberwithin{equation}{section} \theoremstyle{plain}
\newtheorem*{thm*}{Main Theorem}
\newtheorem{thm}{Theorem}[section]
\newtheorem{cor}[thm]{Corollary}
\newtheorem*{cor*}{Corollary}
\newtheorem{lem}[thm]{Lemma}
\newtheorem*{lem*}{Lemma}

\newtheorem*{fact*}{Fact}

\newtheorem*{nota*}{Notation}
\newtheorem{prop}[thm]{Proposition}
\newtheorem*{prop*}{Proposition}

\newtheorem*{rem*}{Remark}

\newtheorem*{observation*}{Observation}

\newtheorem*{exa*}{Example}

\newtheorem*{df*}{Definition}

\newtheorem*{con*}{Construction}

\renewcommand{\geq}{\geqslant}
\renewcommand{\leq}{\leqslant}

\begin{document}
\begin{center}
{\large  \bf  Modules cofinite with respect to ideals of small dimensions}

\vspace{0.5cm} Xiaoyan Yang and Jingwen Shen\\
Department of Mathematics, Northwest Normal University, Lanzhou 730070,
China
E-mails: yangxy@nwnu.edu.cn and 3284541957@qq.com
\end{center}

\bigskip
\centerline { \bf  Abstract}
\leftskip10truemm \rightskip10truemm \noindent Let $\mathfrak{a}$ be an ideal of a noetherian (not necessarily local) ring $R$ and $M$ an $R$-module with $\mathrm{Supp}_RM\subseteq\mathrm{V}(\mathfrak{a})$. We show that if $\mathrm{dim}_RM\leq2$, then $M$ is $\mathfrak{a}$-cofinite if and only if $\mathrm{Ext}^i_R(R/\mathfrak{a},M)$ are finitely generated for all $i\leq 2$, which generalizes one of the main results in [Algebr. Represent. Theory 18 (2015) 369--379]. Some new results concerning cofiniteness of local cohomology modules $\mathrm{H}^i_\mathfrak{a}(M)$ for any finitely generated $R$-module $M$ are obtained.\\
\vbox to 0.3cm{}

\leftskip0truemm \rightskip0truemm
\bigskip
\section* { \bf Introduction and Preliminaries}
Throughout this paper, $R$ is a commutative noetherian ring with identity and $\mathfrak{a}$ an ideal of $R$. For an $R$-module $M$, the $i$th local cohomology of $M$ with respect to $\mathfrak{a}$ is defined as
 \begin{center}$\mathrm{H}^i_\mathfrak{a}(M)=\underrightarrow{\textrm{lim}}_{t>0}\mathrm{Ext}^i_R(R/\mathfrak{a}^t,M)$.\end{center}
The reader can refer to \cite{BS} or \cite{Gr} for more details about local cohomology.

 It is a well-known result that if $(R,\mathfrak{m})$ is a local ring, then
the $R$-module $M$ is artinian if and only if $\mathrm{Supp}_RM\subseteq\{\mathfrak{m}\}$ and $\mathrm{Ext}^i_R(R/\mathfrak{m},M)$ are finitely generated for all $i\geq 0$ (see \cite[Proposition 1.1]{H}).
 In 1968, Grothendieck  \cite{G} conjectured that for any finitely generated $R$-module $M$, the $R$-module $\mathrm{Hom}_R(R/\mathfrak{a},\mathrm{H}^i_\mathfrak{a}(M))$
are finitely generated for
all $i$. One year later, Hartshorne \cite{H} provided a counterexample to show that this conjecture
is false even when $R$ is regular, and then he defined an $R$-module $M$ to be $\mathfrak{a}$-cofinite
if $\mathrm{Supp}_RM\subseteq\mathrm{V}(\mathfrak{a})$ and $\mathrm{Ext}^i_R(R/\mathfrak{a},M)$ is finitely generated for all $i\geq 0$,  and
asked:

\vspace{2mm} \noindent{\bf Question 1.}\label{Th1.4} {\it{Are the local cohomology modules $\mathrm{H}^i_\mathfrak{a}(M)$, $\mathfrak{a}$-cofinite for every finitely generated
$R$-module $M$ and every $i\geq 0$?}}

\vspace{2mm} \noindent{\bf Question 2.}\label{Th1.4} {\it{Is the category $\mathcal{M}(R,\mathfrak{a})_{cof}$ of $\mathfrak{a}$-cofinite $R$-modules an abelian subcategory of the
category $\mathcal{M}(R)$ of $R$-modules?}}
\vspace{2mm}

If the module $\mathrm{H}^i_\mathfrak{a}(M)$ is $\mathfrak{a}$-cofinite, then the set $\mathrm{Ass}_R\mathrm{H}^i_\mathfrak{a}(M)$
of associated primes
and the Bass numbers $\mu^j_{R}(\mathfrak{p},\mathrm{H}^i_\mathfrak{a}(M))$
are finite for all $\mathfrak{p}\in\mathrm{Spec}R$ and $i,j\geq0$. This observation reveals the utmost significance of Question 1.
In the following years, Question 1 were systematically
studied and improved by commutative algebra practitioners in several stages.  For example, see \cite{BN,BNS1,DM,K}.

With respect to the Question 2, Hartshorne with an example showed that this is not true
in general. However, he proved that if $R$ is a
complete regular local ring and $\mathfrak{a}$ a prime ideal with $\mathrm{dim}R/\mathfrak{a}=1$, then the answer to his question is yes. Delfino and Marley \cite{DM}
extended this result to arbitrary complete local rings. Kawasaki \cite{K} generalized the Delfino and Marley's result for an arbitrary
ideal $\mathfrak{a}$ of dimension one in a local ring $R$ by using a spectral sequence argument. In 2014, Bahmanpour et al. \cite{BNS} proved that
Question 2 is true for the category of all $\mathfrak{a}$-cofinite $R$-modules $M$ with $\mathrm{dim}_RM\leq1$
 for all ideals $\mathfrak{a}$ of $R$. The proof of this result is based
on \cite[Proposition 2.6]{BNS} which states that an $R$-module
$M$ with $\mathrm{dim}_RM\leq1$ and $\mathrm{Supp}_RM\subseteq\mathrm{V}(\mathfrak{a})$ is $\mathfrak{a}$-cofinite if and only if
$\mathrm{Hom}_R(R/\mathfrak{a},M)$ and $\mathrm{Ext}^1_R(R/\mathfrak{a},M)$ are finitely generated. Recently, Bahmanpour et al.  \cite{BNS1} extended the above result to the
ideals $\mathfrak{a}$ of dimension two, i.e. $\mathrm{dim}R/\mathfrak{a}=2$, in a local ring $R$.

One of the main goals of present paper is to generalize the Bahmanpour, Naghipour and Sedghi's results to not necessarily local rings. More precisely, we shall prove that:

\vspace{2mm} \noindent{\bf Theorem.}\label{Th1.4} {\it{Let $\mathfrak{a}$ be an ideal of $R$ with $\mathrm{dim}R/\mathfrak{a}\leq2$. Then an $\mathfrak{a}$-torsion $R$-module $M$ is $\mathfrak{a}$-cofinite if and only if $\mathrm{Ext}^i_R(R/\mathfrak{a},M)$ are finitely
generated for all $i\leq2$.}}
\vspace{2mm}

We also answer Question 1 completely in the case
$\mathrm{cd}(\mathfrak{a})\leq2$, and study
 the cofiniteness of modules in $FD_{\leq n}$ for $n=0,1,2$.

Next we recall some notions which we will need later.

We write $\mathrm{Spec}R$ for the set of
prime ideals of $R$ and $\mathrm{Max}R$ for the set of
maximal ideals of $R$, For an ideal $\mathfrak{a}$ in $R$, we set
\begin{center}$\mathrm{V}(\mathfrak{a})=\{\mathfrak{p}\in\textrm{Spec}R\hspace{0.03cm}|\hspace{0.03cm}\mathfrak{a}\subseteq\mathfrak{p}\}$.
\end{center}

Let $M$ be an $R$-module. The set $\mathrm{Ass}_RM$ of associated prime of $M$ is
the set of prime ideals $\mathfrak{p}$ of $R$
such that there exists a cyclic submodule $N$ of $M$ such that $\mathfrak{p}=\mathrm{Ann}_RN$, the annihilator of $N$. The support of an $R$-module $M$ is the set \begin{center}$\mathrm{Supp}_RM=\{\mathfrak{p}\in\mathrm{Spec}R\hspace{0.03cm}|\hspace{0.03cm}
M_\mathfrak{p}\neq0\}$.\end{center}
Recall that a prime ideal $\mathfrak{p}$ of $R$ is said to be an attached prime
of $M$ if $\mathfrak{p}=\mathrm{Ann}_RM/L$ for some submodule $L$ of $M$ (see \cite{MS}). The set of attached
primes of $M$ is denoted by $\mathrm{Att}_RM$. If $M$ is artinian, then $M$ admits a
minimal secondary representation $M=M_1+\cdots+M_r$ so that $M_i$
is $\mathfrak{p}_i$-secondary for $i=1,\cdots,r$. In this case, $\mathrm{Att}_RM=\{\mathfrak{p}_1,\cdots,\mathfrak{p}_r\}$.

The arithmetic
rank of the ideal $\mathfrak{a}$, denoted by
$\mathrm{ara}(\mathfrak{a})$, is the least number of elements of $R$ required to generate an ideal which has
the same radical as $\mathfrak{a}$, i.e.,
\begin{center}$\mathrm{ara}(\mathfrak{a})=\mathrm{min}\{n\geq0\hspace{0.03cm}|\hspace{0.03cm}\exists\ a_1,\cdots,a_n\in R\ \textrm{with}\ \mathrm{Rad}(a_1,\cdots,a_n)=\mathrm{Rad}(\mathfrak{a})\}$.\end{center} The arithmetic rank of $\mathfrak{a}$ in $R$ with respect
to an $R$-module $M$, denoted by $\mathrm{ara}_M(\mathfrak{a})$, is defined by the arithmetic rank of the ideal $\mathfrak{a}+
\mathrm{Ann}_RM/\mathrm{Ann}_RM$ in the ring $R/\mathrm{Ann}_RM$.

For each
$R$-module $M$, set
\begin{center}$\mathrm{cd}(\mathfrak{a},M)=\mathrm{sup}\{n\in\mathbb{Z}\hspace{0.03cm}|\hspace{0.03cm}\mathrm{H}_\mathfrak{a}^n(M)\neq0\}$.\end{center}
The cohomological dimension of the ideal $\mathfrak{a}$ is
\begin{center}$\mathrm{cd}(\mathfrak{a})=\mathrm{sup}\{\mathrm{cd}(\mathfrak{a},M)\hspace{0.03cm}|\hspace{0.03cm}M\ \textrm{is\ an}\ R\textrm{-module}\}$.\end{center}

\bigskip
\section{\bf Main resulats}
One of the main results of this work is
examine the cofiniteness of modules with respect to ideals of dimension 2 in an arbitrary noetherian ring,
which is a generalization of \cite[Proposition 2.6]{BNS} and \cite[Theorem 3.5]{BNS1}.

\begin{thm}\label{lem:2.1}{\it{Let $M$ be an $\mathfrak{a}$-torsion $R$-module with $\mathrm{dim}_RM\leq2$. Then
 $M$ is $\mathfrak{a}$-cofinite if and only if $\mathrm{Hom}_R(R/\mathfrak{a},M)$, $\mathrm{Ext}^1_R(R/\mathfrak{a},M)$ and $\mathrm{Ext}^2_R(R/\mathfrak{a},M)$ are finitely generated.}}
\end{thm}
\begin{proof} ``Only if'' part is obvious.

``If'' part. By \cite[Proposition 2.6]{BNS} we may assume $\mathrm{dim}_RM=2$ and let
$t=\mathrm{ara}_M(\mathfrak{a})$. If $t=0$, then $\mathfrak{a}^n\subseteq\mathrm{Ann}_RM$ for
some $n\geq0$ by definition, and so $M=(0:_M\mathfrak{a}^n)$ is finitely generated by \cite[Lemma 2.1]{ANS} and the assertion follows. Next, assume that $t>0$. Let\begin{center}
$T_M=\{\mathfrak{p}\in\mathrm{Supp}_RM\hspace{0.03cm}|\hspace{0.03cm}\mathrm{dim}R/\mathfrak{p}=2\}$.
\end{center}
As $\mathrm{Ass}_R\mathrm{Hom}_R(R/\mathfrak{a},M)=\mathrm{Ass}_RM$ is finite, the set $T_M$ is finite. Also for each $\mathfrak{p}\in T_M$, the $R_\mathfrak{p}$-module $\mathrm{Hom}_{R_\mathfrak{p}}(R_\mathfrak{p}/\mathfrak{a}R_\mathfrak{p},M_\mathfrak{p})$ is finitely generated and $M_\mathfrak{p}$ is an $\mathfrak{a}R_\mathfrak{p}$-torsion $R_\mathfrak{p}$-module with $\mathrm{Supp}_{R_\mathfrak{p}}M_\mathfrak{p}\subseteq\mathrm{V}(\mathfrak{p}R_\mathfrak{p})$, it follows from \cite[Proposition 4.1]{LM}
that $M_\mathfrak{p}$ is an artinian $\mathfrak{a}R_\mathfrak{p}$-torsion $R_\mathfrak{p}$-module. Let
$T_M=\{\mathfrak{p}_1,\cdots\mathfrak{,p}_n\}$.
It follows from \cite[Lemma 2.5]{BN} that $\mathrm{V}(\mathfrak{a}R_{\mathfrak{p}_j})\cap\mathrm{Att}_{R_{\mathfrak{p}_j}}M_{\mathfrak{p}_j}
\subseteq\mathrm{V}(\mathfrak{p}_jR_{\mathfrak{p}_j})$ for $j=1,\cdots,n$. Let \begin{center}
$\mathrm{U}_M=\bigcup_{j=1}^n\{\mathfrak{q}\in\mathrm{Spec}R\hspace{0.03cm}|\hspace{0.03cm}\mathfrak{q}R_{\mathfrak{p}_j}
\in\mathrm{Att}_{R_{\mathfrak{p}_j}}M_{\mathfrak{p}_j}\}$.
\end{center}Then $\mathrm{U}_M\cap\mathrm{V}(\mathfrak{a})\subseteq T_M$. Since $t=\mathrm{ara}_M(\mathfrak{a})\geq 1$, there exist $y_1,\cdots,y_t\in\mathfrak{a}$
such that\begin{center}
$\mathrm{Rad}(\mathfrak{a}+\mathrm{Ann}_RM/\mathrm{Ann}_RM)=\mathrm{Rad}((y_1,\cdots,y_t)+\mathrm{Ann}_RM/\mathrm{Ann}_RM)$.
\end{center}Since $\mathfrak{a}\nsubseteq\bigcup_{\mathfrak{q}\in\mathrm{U}_M\backslash\mathrm{V}(\mathfrak{a})}\mathfrak{q}$, it follows that $(y_1,\cdots,y_t)+\mathrm{Ann}_RM\nsubseteq\bigcup_{\mathfrak{q}\in\mathrm{U}_M\backslash\mathrm{V}(\mathfrak{a})}\mathfrak{q}$.
On the other hand, for each $\mathfrak{q}\in\mathrm{U}_M$ we have $\mathfrak{q}R_{\mathfrak{p}_j}
\in\mathrm{Att}_{R_{\mathfrak{p}_j}}M_{\mathfrak{p}_j}$ for some $1\leq j\leq n$. Thus
\begin{center}
$(\mathrm{Ann}_RM)R_{\mathfrak{p}_j}\subseteq\mathrm{Ann}_{R_{\mathfrak{p}_j}}M_{\mathfrak{p}_j}\subseteq \mathfrak{q}R_{\mathfrak{p}_j}$,
\end{center}and so $\mathrm{Ann}_RM\subseteq\mathfrak{q}$. Consequently, $(y_1,\cdots,y_t)\nsubseteq\bigcup_{\mathfrak{q}\in\mathrm{U}_M\backslash\mathrm{V}(\mathfrak{a})}\mathfrak{q}$ as $\mathrm{Ann}_RM\subseteq\bigcap_{\mathfrak{q}\in\mathrm{U}_M\backslash\mathrm{V}(\mathfrak{a})}\mathfrak{q}$. Hence \cite[Ex.16.8]{M}
provides an element $a_1\in(y_2,\cdots,y_t)$ such that $y_1+a_1\not\in\bigcup_{\mathfrak{q}\in\mathrm{U}_M\backslash\mathrm{V}(\mathfrak{a})}\mathfrak{q}$. Set $x_1=y_1+a_1$. Then $x_1\in\mathfrak{a}$ and
\begin{center}
$\mathrm{Rad}(\mathfrak{a}+\mathrm{Ann}_RM/\mathrm{Ann}_RM)=\mathrm{Rad}((x_1,y_2,\cdots,y_t)+\mathrm{Ann}_RM/\mathrm{Ann}_RM)$.
\end{center}Let $M_1=(0:_Mx_1)$. Then $\mathrm{ara}_{M_1}(\mathfrak{a})=
\mathrm{Rad}((y_2,\cdots,y_t)+\mathrm{Ann}_RM_1/\mathrm{Ann}_RM_1)\leq t-1$ as $x_1\in\mathrm{Ann}_RM_1$. The reasoning in the preceding  applied to $M_1$, there exists $x_2\in\mathfrak{a}$ such that $\mathrm{ara}_{M_2}(\mathfrak{a})=
\mathrm{Rad}((y_3,\cdots,y_t)+\mathrm{Ann}_RM_2/\mathrm{Ann}_RM_2)\leq t-2$.
 Continuing this process, one obtains elements $x_1,\cdots,x_t\in\mathfrak{a}$ and the sequences
\begin{center}
$0\rightarrow M_{i}\rightarrow M_{i-1}\rightarrow x_iM_{i-1}\rightarrow0$,
\end{center}where $M_{0}=M$ and $M_{i}=(0:_{M_{i-1}}x_i)$ such that $\mathrm{ara}_{M_i}(\mathfrak{a})\leq t-i$ for $i=1,\cdots,t$, and these
exact sequences induce an exact sequence
\begin{center}
$0\rightarrow M_t\rightarrow M\rightarrow (x_1,\cdots,x_t)M\rightarrow0$.
\end{center}As $\mathrm{Hom}_R(R/\mathfrak{a},M_t)$ is finitely generated and $\mathrm{ara}_{M_t}(\mathfrak{a})=0$, one has $M_t$ is $\mathfrak{a}$-cofinite. Moreover, the above sequence implies that
$\mathrm{Ext}^1_R(R/\mathfrak{a},(x_1,\cdots,x_t)M)$ and $\mathrm{Ext}^2_R(R/\mathfrak{a},(x_1,\cdots,x_t)M)$ are finitely generated. Also the exact sequence \begin{center}$0\rightarrow (x_1,\cdots,x_t)M\rightarrow M\rightarrow M/(x_1,\cdots,x_t)M\rightarrow0$ \end{center} yields that $\mathrm{Hom}_R(R/\mathfrak{a},M/(x_1,\cdots,x_t)M)$ and $\mathrm{Ext}^1_R(R/\mathfrak{a},M/(x_1,\cdots,x_t)M)$ are finitely generated. On the other hand, for $i=1,\cdots,t$, the exact sequence \begin{center}
$0\rightarrow x_iM_{i-1}\rightarrow (x_1,\cdots,x_i)M\rightarrow (x_1,\cdots,x_{i-1})M\rightarrow0$
\end{center}induces the following commutative diagram
\begin{center}$\xymatrix@C=23pt@R=20pt{
     &  0\ar[d]  & 0 \ar[d] &  \\
       0\ar[r]&x_iM_{i-1}\ar[r] \ar[d]& M_{i-1}\ar[d] \ar[r]& M_{i-1}/x_iM_{i-1}\ar[d]^\cong\ar[r]&0\\
      0 \ar[r] &(x_1,\cdots,x_i)M \ar[d] \ar[r] &M  \ar[d]\ar[r] &M/(x_1,\cdots,x_i)M  \ar[r] & 0  \\
     &    (x_1,\cdots,x_{i-1})M  \ar[d]\ar@{=}[r] & (x_1,\cdots,x_{i-1})M\ar[d]  \\
    &  0& \hspace{0.15cm}0.  &}$
\end{center}Let $T_{M_{t-1}}=\{\mathfrak{p}\in\mathrm{Supp}_RM_{t-1}\hspace{0.03cm}|\hspace{0.03cm}\mathrm{dim}R/\mathfrak{p}=2\}
=\{\mathfrak{q}_{1},\cdots,\mathfrak{q}_{m}\}$. Since \begin{center}$\mathrm{Hom}_R(R/\mathfrak{a},M_{t-1}/x_tM_{t-1})\cong\mathrm{Hom}_R(R/\mathfrak{a},M/(x_1,\cdots,x_t)M)$\end{center} is finitely generated, it follows from \cite[Lemma 2.4]{BN} that $(M_{t-1}/x_tM_{t-1})_{\mathfrak{q}_{j}}$ has finite length for $j=1,\cdots,m$. Thus there is a finitely generated submodule $L_{j}$ of $M/(x_1,\cdots,x_t)M$ such that $(M/(x_1,\cdots,x_t)M)_{\mathfrak{q}_{j}}=(L_{j})_{\mathfrak{q}_{j}}$. Set $L=L_{1}+\cdots+L_{m}$. Then $L$ is a finitely generated submodule of $M/(x_1,\cdots,x_t)M$ so that $\mathrm{Supp}_R(M/(x_1,\cdots,x_t)M)/L\subseteq\mathrm{Supp}_RM_{t-1}\backslash\{\mathfrak{q}_{1},\cdots,\mathfrak{q}_{m}\}$ and hence $\mathrm{dim}_R(M/(x_1,\cdots,x_t)M)/L\leq1$. Now the exact sequence \begin{center}$0\rightarrow L\rightarrow M/(x_1,\cdots,x_t)M\rightarrow (M/(x_1,\cdots,x_t)M)/L\rightarrow0$\end{center}induces the following exact sequence
\begin{center}
$\mathrm{Hom}_{R}(R/\mathfrak{a},M/(x_1,\cdots,x_t)M)\rightarrow
\mathrm{Hom}_{R}(R/\mathfrak{a},(M/(x_1,\cdots,x_t)M)/L)\rightarrow\mathrm{Ext}^1_{R}(R/\mathfrak{a},L)\rightarrow
\mathrm{Ext}^1_{R}(R/\mathfrak{a},M/(x_1,\cdots,x_t)M)\rightarrow\mathrm{Ext}^1_{R}(R/\mathfrak{a},(M/(x_1,\cdots,x_t)M)/L)
\rightarrow\mathrm{Ext}^2_{R}(R/\mathfrak{a},L)$.\end{center} Hence $\mathrm{Hom}_{R}(R/\mathfrak{a},(M/(x_1,\cdots,x_t)M)/L)$ and $\mathrm{Ext}^1_{R}(R/\mathfrak{a},(M/(x_1,\cdots,x_t)M)/L)$ are finitely generated, and so $(M/(x_1,\cdots,x_t)M)/L$ is $\mathfrak{a}$-cofinite by \cite[Proposition 2.6]{BNS}. Consequently, $M/(x_1,\cdots,x_t)M$ is $\mathfrak{a}$-cofinite. As $M_t\cong\mathrm{Hom}_{R}(R/(x_1,\cdots,x_t)R,M)$ and $M/(x_1,\cdots,x_t)M$ are $\mathfrak{a}$-cofinite, it follows from \cite[Corollary 3.3]{LM} that $M$ is $\mathfrak{a}$-cofinite, as desired.
\end{proof}

\begin{cor}\label{lem:2.0}{\it{Let $\mathfrak{a}$ be a proper ideal of $R$ with $\mathrm{dim}R/\mathfrak{a}\leq2$ and $M$ an $R$-module such that $\mathrm{Supp}_RM\subseteq\mathrm{V}(\mathfrak{a})$. Then the
 following are equivalent:

$(1)$ $M$ is $\mathfrak{a}$-cofinite;

$(2)$ $\mathrm{Hom}_R(R/\mathfrak{a},M)$, $\mathrm{Ext}^1_R(R/\mathfrak{a},M)$ and $\mathrm{Ext}^2_R(R/\mathfrak{a},M)$ are finitely generated.}}
\end{cor}

\begin{cor}\label{lem:2.2}{\it{Let $\mathfrak{a}$ be a proper ideal of $R$ with $\mathrm{dim}R/\mathfrak{a}=2$, and let $M$ be a finitely generated $R$-module. Then $\mathrm{H}^i_\mathfrak{a}(M)$ is $\mathfrak{a}$-cofinite for every $i\in\mathbb{Z}$ if and only if $\mathrm{Hom}_R(R/\mathfrak{a},\mathrm{H}^i_\mathfrak{a}(M))$ is finitely
generated for every $i\in\mathbb{Z}$.}}
\end{cor}
\begin{proof}  This follows from \cite[Theorem 2.9]{BA} and Theorem \ref{lem:2.1}.
\end{proof}

\begin{cor}\label{lem:2.3}{\it{Let $\mathfrak{a}$ be a proper ideal of $R$ such that $\mathrm{dim}R/\mathfrak{a}=2$, and let $M$ be a finitely generated $R$-module.

$(1)$ If $\mathrm{H}^{2i}_\mathfrak{a}(M)$ is $\mathfrak{a}$-cofinite for every $i\geq0$, then $\mathrm{H}^i_\mathfrak{a}(M)$ is $\mathfrak{a}$-cofinite for every $i\geq0$.

$(2)$ If $\mathrm{H}^{2i+1}_\mathfrak{a}(M)$ is $\mathfrak{a}$-cofinite for every $i\geq0$, then $\mathrm{H}^i_\mathfrak{a}(M)$ is $\mathfrak{a}$-cofinite for every $i\geq0$.}}
\end{cor}

 We are now ready to answer Question 1 subsequently in the case $\mathrm{cd}(\mathfrak{a})=2$, which is a generalization of \cite[Theorem 7.4]{LM}.

\begin{lem}\label{lem:4.2}{\it{Let $\mathfrak{a}$ be a proper ideal of $R$ with $\mathrm{ara}(\mathfrak{a})\leq2$. Then $\mathrm{H}^i_\mathfrak{a}(M)$ are $\mathfrak{a}$-cofinite for all $i\in\mathbb{Z}$ and all finitely generated $R$-modules $M$.}}
\end{lem}
\begin{proof} Note that $\mathrm{H}^i_\mathfrak{a}(M)=0$ for all $i\geq \mathrm{ara}(\mathfrak{a})$. If $\mathrm{ara}(\mathfrak{a})=0$, then the result obviously holds. If $\mathrm{ara}(\mathfrak{a})=1$, then $\mathrm{H}^0_\mathfrak{a}(M)$ is $\mathfrak{a}$-cofinite. Hence \cite[Proposition 3.11]{LM} implies that $\mathrm{H}^1_\mathfrak{a}(M)$ is $\mathfrak{a}$-cofinite. Now suppose that $\mathrm{ara}(\mathfrak{a})=2$. Then $\mathrm{Rad}(\mathfrak{a})=\mathrm{Rad}(a_1,a_2)$ where $a_1,a_2\in R$. Set $\mathfrak{b}=(a_1)$ and $\mathfrak{c}=(a_2)$. By Mayer-Vietoris sequence, one has an exact sequence\begin{center}$\mathrm{H}^1_{\mathfrak{b}\cap\mathfrak{c}}(M)\rightarrow \mathrm{H}^2_\mathfrak{a}(M)\rightarrow\mathrm{H}^2_\mathfrak{b}(M)\oplus\mathrm{H}^2_\mathfrak{c}(M)=0$.\end{center}By the preceding proof, $\mathrm{H}^1_{\mathfrak{b}\cap\mathfrak{c}}(M)$ is $(\mathfrak{b}\cap\mathfrak{c})$-cofinite and so  $\mathrm{Ext}^i_R(R/\mathfrak{a},\mathrm{H}^1_{\mathfrak{b}\cap\mathfrak{c}}(M))$ are finitely generated for all $i\geq 0$ by \cite[Proposition 7.2]{WW}. Consequently, the above exact sequence implies that $\mathrm{Ext}^i_R(R/\mathfrak{a},\mathrm{H}^2_{\mathfrak{a}}(M))$ are finitely generated for all $i\geq 0$. Hence $\mathrm{H}^i_\mathfrak{a}(M)$ are $\mathfrak{a}$-cofinite for all $i$ by \cite[Proposition 3.11]{LM} again.
\end{proof}

\begin{thm}\label{lem:4.4}{\it{Let $\mathfrak{a}$ be a proper ideal of $R$ with $\mathrm{cd}(\mathfrak{a})\leq2$. Then $\mathrm{H}^i_\mathfrak{a}(M)$ is $\mathfrak{a}$-cofinite for all $i$ and all finitely generated $R$-modules $M$.}}
\end{thm}
\begin{proof} When $\mathrm{cd}(\mathfrak{a})=0,1$, there is nothing to prove. So suppose that $\mathrm{cd}(\mathfrak{a})=2$. We argue by induction on $d=\mathrm{ara}(\mathfrak{a})\geq2$. If $d=2$ then the result holds by Lemma \ref{lem:4.2}. Now suppose $d>2$ and the result has been proved for smaller values of
$d$. Let $a_1,\cdots,a_d\in R$ with $\mathrm{Rad}(\mathfrak{a})=\mathrm{Rad}(a_1,\cdots,a_d)$, and set $\mathfrak{b}=(a_1,\cdots,a_{d-1})$ and $\mathfrak{c}=(a_d)$. By Mayer-Vietoris sequence, one has the following exact sequence\begin{center}$\mathrm{H}^1_{\mathfrak{b}\cap\mathfrak{c}}(M)\rightarrow \mathrm{H}^2_\mathfrak{a}(M)\rightarrow\mathrm{H}^2_\mathfrak{b}(M)$.\end{center}By the induction, $\mathrm{H}^1_{\mathfrak{b}\cap\mathfrak{c}}(M)$ is $\mathfrak{b}\cap\mathfrak{c}$-cofinite and $\mathrm{H}^2_{\mathfrak{b}}(M)$ is $\mathfrak{b}$-cofinite, so $\mathrm{Ext}^i_R(R/\mathfrak{a},\mathrm{H}^1_{\mathfrak{b}\cap\mathfrak{c}}(M))$ and $\mathrm{Ext}^i_R(R/\mathfrak{a},\mathrm{H}^2_{\mathfrak{b}}(M))$ are finitely generated for all $i\geq 0$. Thus the above exact sequence implies that $\mathrm{Ext}^i_R(R/\mathfrak{a},\mathrm{H}^2_{\mathfrak{a}}(M))$ are finitely generated for all $i\geq 0$, and therefore $\mathrm{H}^i_\mathfrak{a}(M)$ are $\mathfrak{a}$-cofinite for all $i\in\mathbb{Z}$.
\end{proof}

\begin{cor}\label{lem:4.5}{\it{Let $R$ be a ring with $\mathrm{dim}R\leq2$. Then $\mathrm{H}^i_\mathfrak{a}(M)$ is $\mathfrak{a}$-cofinite for all $i$ and all finitely generated $R$-modules $M$.}}
\end{cor}
\begin{proof} This follows from that $\mathrm{cd}(\mathfrak{a})\leq\mathrm{dim}R$.
\end{proof}

Let $n\geq-1$ be an integer. Recall that an $R$-module $M$ is said to be $FD_{\leq n}$
if there is a finitely generated submodule $N$ of $M$ such that $\mathrm{dim}_RM/N\leq n$.
By definition, any finitely generated $R$-module and any $R$-module with dimension
at most $n$ are $FD_{\leq n}$.

Next, we
study the cofiniteness of modules in $FD_{\leq n}$ for $n=0,1,2$, the results generalize \cite[Corollary 2.6]{BNS1} and \cite[Theorem 2.6]{BN}.

\begin{prop}\label{lem:3.6}{\it{Let $n=0,1,2$, and let $M$ be an $R$-module in $FD_{\leq n}$ with $\mathrm{Supp}_RM\subseteq\mathrm{V}(\mathfrak{a})$. Then
 $M$ is $\mathfrak{a}$-cofinite if and only if $\mathrm{Ext}^i_R(R/\mathfrak{a},M)$ are finitely generated for all $i\leq n$.}}
\end{prop}
\begin{proof} ``Only if'' part is obvious.

``If'' part. Let $N$ be a finitely generated submodule of $M$ such that $\mathrm{dim}_RM/N\leq n$. Then the exact sequence
\begin{center}
$\mathrm{Hom}_R(R/\mathfrak{a},M)\rightarrow\mathrm{Hom}_R(R/\mathfrak{a},M/N)
\rightarrow\mathrm{Ext}^1_R(R/\mathfrak{a},N)\rightarrow\mathrm{Ext}^1_R(R/\mathfrak{a},M)\rightarrow \mathrm{Ext}^1_R(R/\mathfrak{a},M/N)\rightarrow\mathrm{Ext}^2_R(R/\mathfrak{a},N)\rightarrow\mathrm{Ext}^2_R(R/\mathfrak{a},M)\rightarrow \mathrm{Ext}^2_R(R/\mathfrak{a},M/N)\rightarrow\mathrm{Ext}^3_R(R/\mathfrak{a},N)$
\end{center}implies that $\mathrm{Ext}^i_R(R/\mathfrak{a},M/N)$ are finitely generated for all $i\leq n$. Hence $M/N$ is $\mathfrak{a}$-cofinite by \cite[Proposition 4.1]{LM}, \cite[Proposition 2.6]{BNS} and Theorem \ref{lem:2.1} and then $M$ is $\mathfrak{a}$-cofinite.
\end{proof}

An $R$-module $M$ is minimax if there is a
finitely generated submodule $N$ of $M$, such that $M/N$ is artinian.

\begin{cor}\label{lem:3.6'}{\rm (\cite[Proposition 4.3]{LM}.)} {\it{Let $M$ be a minimax $R$-module with $\mathrm{Supp}_RM\subseteq\mathrm{V}(\mathfrak{a})$. Then
 $M$ is $\mathfrak{a}$-cofinite if and only if $\mathrm{Hom}_R(R/\mathfrak{a},M)$ is finitely generated.}}
\end{cor}

An $R$-module $M$ is said to be weakly Laskerian if the set $\mathrm{Ass}_RM/N$
is finite for each submodule $N$ of $M$.

\begin{cor}\label{lem:3.7'} {\it{Let $M$ be a weakly Laskerian $R$-module with $\mathrm{Supp}_RM\subseteq\mathrm{V}(\mathfrak{a})$. Then
 $M$ is $\mathfrak{a}$-cofinite if and only if $\mathrm{Hom}_R(R/\mathfrak{a},M)$ and $\mathrm{Ext}^1_R(R/\mathfrak{a},M)$ are finitely generated.}}
\end{cor}
\begin{proof}  This follows from \cite[Theorem 3.3]{KB} that $M$ is $FD_{\leq 1}$.
\end{proof}

\begin{cor}\label{lem:3.7}{\it{Let $n=0,1$, and let $M$ be a non-zero finitely generated $R$-module and $t$ a non-negative integer such that $\mathrm{H}^i_\mathfrak{a}(M)$ are $FD_{\leq n}$ for all $i<t$. Then

$(1)$ $\mathrm{H}^i_\mathfrak{a}(M)$ are $\mathfrak{a}$-cofinite for all $i=0,\cdots,t-1$;

$(2)$ $\mathrm{Ext}^j_R(R/\mathfrak{a},\mathrm{H}^t_\mathfrak{a}(M))$ are finitely generated for all $j\leq n$.}}
\end{cor}
\begin{proof} We use induction on $t$.
If $t=1$ then $\mathrm{Ext}^j_R(R/\mathfrak{a},\mathrm{H}^0_\mathfrak{a}(M))$ are finitely generated for all $j\leq n$ by \cite[Theorem 2.9]{BA}. Hence $\mathrm{H}^0_\mathfrak{a}(M)$ is $\mathfrak{a}$-cofinite by Proposition \ref{lem:3.6} and $\mathrm{Ext}^j_R(R/\mathfrak{a},\mathrm{H}^1_\mathfrak{a}(M))$ are finitely generated for all $j\leq n$ by \cite[Theorem 2.9]{BA}.
  Now suppose that $t>1$ and that the case $t-1$ is settled.
 The induction implies that $\mathrm{H}^i_\mathfrak{a}(X)$ is $\mathfrak{a}$-cofinite for $i<t-1$, and so $\mathrm{Ext}^j_R(R/\mathfrak{a},\mathrm{H}^{t-1}_\mathfrak{a}(M))$ are finitely generated for all $j\leq n$ by \cite[Theorem 2.9]{BA}. Consequently, $\mathrm{H}^{t-1}_\mathfrak{a}(X)$ is $\mathfrak{a}$-cofinite by Proposition \ref{lem:3.6} and $\mathrm{Ext}^j_R(R/\mathfrak{a},\mathrm{H}^t_\mathfrak{a}(M))$ are finitely generated for all $j\leq n$ by \cite[Theorem 2.9]{BA} again.
\end{proof}

\bigskip \centerline {\bf ACKNOWLEDGEMENTS} This research was partially supported by National Natural Science Foundation of China (11761060,11901463).

\end{document}